\documentclass[a4paper,12pt]{article}

\usepackage{amscd}
\usepackage{amsfonts}
\usepackage{amsmath}
\usepackage{amssymb}
\usepackage{amsthm}
\usepackage{ascmac}
\usepackage{here}
\usepackage{makeidx}
\usepackage{mathrsfs}
\usepackage{slashbox}
\usepackage{txfonts}

\allowdisplaybreaks
\makeindex

\newcommand{\N}{\mathbb{N}}
\newcommand{\Z}{\mathbb{Z}}

\newcommand{\Q}{\mathbb{Q}}
\newcommand{\R}{\mathbb{R}}
\newcommand{\C}{\mathbb{C}}

\newcommand{\M}{\mathscr{M}}

\newcommand{\U}{\mathscr{U}}
\newcommand{\V}{\mathscr{V}}

\newtheorem{theorem}{Theorem}[section]
\newtheorem{definition}[theorem]{Definition}
\newtheorem{lemma}[theorem]{Lemma}
\newtheorem{proposition}[theorem]{Proposition}
\newtheorem{corollary}[theorem]{Corollary}
\newtheorem{remark}[theorem]{Remark}

\def\set#1#2{\left\{\ #1 \ \middle|\ #2 \ \right\}}
\def\m#1{\mbox{#1}}
\def\ms#1{\mbox{\scriptsize #1}}
\def\mt#1{\mbox{\tiny #1}}

\title{Set Theory and $p$-adic Algebras}
\author{University of Tokyo, Tomoki Mihara}
\date{}

\begin{document}

\maketitle
\tableofcontents

\section{Introduction}
\label{Introduction}

In our previous paper, \cite{Mih}, we studied the property of the universal totally disconnected Hausdorff compactification $\iota_{\ms{TD},X} \colon X \to \m{TD}(X)$ of a topological space $X$ using the Berkovich spectrum of the Banach $k$-algebra of $k$-valued bounded continuous functions on $X$ for an arbitrary non-Archimedean field $k$. As a consequence, we have obtained the method to analyse the universal totally disconnected Hausdorff compactification $\m{TD}(X)$ and the totally disconnected boundary $\partial_{\ms{TD}} X \coloneqq \m{TD}(X) \backslash \iota_{\ms{TD},X}(X)$ applying the elementary $p$-adic analysis to them. The aim of this paper is to verify it is independent of the axiom of $\m{ZFC}$ that the $k$-algebra $\m{C}(\partial_{\ms{TD}} X,k)$ of $k$-valued continuous functions on the totally disconnected boundary $\partial_{\ms{TD}} X$ has a maximal ideal of height $0$ for any separated totally disconnected non-compact Hausdorff topological space $X$ countable at infinity and any local field $k$. This is Theorem \ref{main theorem}. The proof of the independence depends on Shelah's theory on a proper forcing, \cite{She}.

Note that our theory is an analogue of the result for the Stone-$\check{\m{C}}$ech compactification $\iota_{\beta,X} \colon X \to \beta X$ of a topological space $X$. Under the continuum hypothesis, the maximal boundary $\partial_{\beta} X \colon \beta X \backslash \iota_{\beta,X}(X)$ of $X$ has a P-point if $X$ is a separable non-compact Hausdorff topological space countable at infinity. The proof heavily relies on the real analysis and the lifting property for an $\R$-valued bounded continuous function on $X$, and hence a similar method does not work for another compactification of $X$ than the Stone-$\check{\m{C}}$ech compactification $\iota_{\beta,X} \colon X \to \beta X$, such as the universal totally disconnected Hausdorff compactification $\iota_{\ms{TD},X} \colon X \to \m{TD}(X)$. The key idea of our proof is the alteration of the base field. The universal totally disconnected Hausdorff compactification $\iota_{\ms{TD},X} \colon X \to \m{TD}(X)$ satisfies the lifting property for a $k$-valued bounded continuous function on $X$ when the base field $k$ is a finite field or a local field, and hence the $p$-adic analysis is valid here instead of the real analysis.

\section{P-point}
\label{P-point}

Throughout this paper, we adopt the axiom of Zermelo--Fraenkel set theory with the axiom of choice, and hence Zorn's lemma and the well-ordering principle hold. We write $\m{ZFC}$ instead of Zermelo--Fraenkel set theory with the axiom of choice for short. We review the topological foundation first. A P-point of a topological space corresponds to a maximal ideal of height $0$ in the Banach algebra of bounded continuous functions over a non-Archimedean field, and hence the existence of a P-point is deeply related on the ring structure of such a ring. Throughout this paper, denote by $X$ a topological space and by $k$ a complete valuation field.

\begin{definition}
Denote by $\m{CO}(X) \subset 2^X$ the set of clopen subsets of $X$, and by $\m{CO}_{\delta}(X) \subset 2^X$ the set of countable intersections of clopen subsets of $X$. A subset $F \subset X$ is said to be a $\m{CO}_{\delta}$-subset of $X$ if $F \in \m{CO}_{\delta}(X)$. For a point $x \in X$, denote by $N_{X,x} \subset 2^X$ the subset of neighbourhoods of $x$ in $X$.
\end{definition}

\begin{lemma}
\label{zero set}
Suppose the valuation of $k$ is non-trivial. A subset $F \subset X$ is a $\m{CO}_{\delta}$-subset of $X$ if and only if there is a $k$-valued bounded continuous function $f \colon X \to k$ such that $F = f^{-1}(\{ 0 \})$.
\end{lemma}

\begin{proof}
Suppose there is a $k$-valued bounded continuous function $f \colon X \to k$ such that $F = f^{-1}(\{ 0 \})$. For a non-negative integer $n \in \N$, set
\begin{eqnarray*}
  U_n \coloneqq \set{x \in X}{|f(x)| < (n + 1)^{-1}}.
\end{eqnarray*}
By the continuity of $f$, $U_n$ is open in $X$. Take a point $x \in X \backslash U_n$. Since $f$ is continuous, there is an open neighbourhood $V \subset X$ of $x$ in $X$ such that $|f(y) - f(x)| < (n + 1)^{-1}$. Since $x \in X \backslash U_n$, one has $|f(x)| \geq (n + 1)^{-1}$ and hence $|f(y)| = |(f(y) - f(x)) + f(x)| = |f(x)|$ for any $y \in V$. It follows $V \subset X \backslash U_n$, and therefore $X \backslash U_n$ is open. One has obtained a decreasing sequence $U_0 \subset U_1 \subset \cdots \subset U_n \subset \cdots$ of clopen subsets of $X$, and it is obvious that $\bigcap_{n \in \N} U_n = f^{-1}(\{ 0 \}) = F$.

On the other hand, suppose $F$ is a $\m{CO}_{\delta}$-subset of $X$. Take a countable collection $\U \subset \m{CO}(X)$ of clopen subsets of $X$ with $\bigcap \U = F$. If $|\U| < \aleph_0$, then $F$ is a clopen subset of $X$ and hence the $k$-valued characteristic function $1_{X \backslash F} \colon X \to k$ on the complement $X \backslash F \in \m{CO}(X)$ is a $k$-valued bounded continuous function with $1_{X \backslash F}^{-1}(\{ 0 \}) = F$. Otherwise, one has $|\U| = \aleph_0$ by the assumption of the countability of $\U$, and take a bijective map $\phi \colon \N \to \U$. For a non-negative integer $n \in \N$, set $U_n \coloneqq \phi(0) \cap \cdots \cap \phi(n) \in \m{CO}(X)$. It is obvious that $U_0 \subset U_1 \subset \cdots \subset U_n \subset \cdots \subset F$ and $\bigcap_{n \in \N} U_n = F$. Since the valuation of $k$ is not trivial, there is an element $a \in k^{\times}$ such that $0 < |a| < 1$. Define the function $f \colon X \to k$ as the infinite sum
\begin{eqnarray*}
  f \coloneqq \sum_{n \in \N} a^n 1_{X \backslash U_n} \colon X & \to & k \\
  x & \mapsto & \sum_{n \in \N} a^n 1_{X \backslash U_n}(x),
\end{eqnarray*}
which is a $k$-valued bounded continuous function on $X$ because it is locally constant and one has
\begin{eqnarray*}
  && |f(x)| = \left| \sum_{n \in \N} a^n 1_{X \backslash U_n}(x) \right| = \left| \sum_
  {
    \ms{
    $
    \begin{array}{c}
      n \in \N \\
      x \notin U_n
    \end{array}
    $
    }
  }
  a^n \right| =
  \left\{
    \begin{array}{ll}
      |a|^n & (x \in U_{n - 1} \backslash U_n, {}^{\exists} n \in \N) \\
      0 & (x \in \bigcap_{n \in \N} U_n = F),
    \end{array}
  \right.
\end{eqnarray*}
for any $x \in X$, where we formally set $U_{-1} \coloneqq X$. Thus one obtains a $k$-valued bounded continuous function $f \colon X \to k$ with $f^{-1}(\{ 0 \}) = F$.
\end{proof}

\begin{definition}
A point $x \in X$ is said to be a P-point of $X$ if for any countable family $\U \subset N_{X,x}$ of neighbourhoods of $x$ in $X$, the intersection of $\bigcap \U \subset X$ is contained in $N_{X,x}$. In other words, $x \in X$ is a P-point if and only if $N_{X,x}$ is closed under countable intersections.
\end{definition}

\begin{lemma}
\label{zero dimensional}
Suppose $X$ is a $T_1$ topological space. If the set $\m{CO}(X)$ of clopen subsets forms a basis of $X$, then $X$ is totally disconnected. Conversely if $X$ is a totally disconnected locally compact Hausdorff topological space, then $\m{CO}(X)$ forms a basis of $X$.
\end{lemma}

\begin{proof}
Suppose $\m{CO}(X)$ forms a basis of $X$. Take a connected component $C \subset X$ of $X$, and assume $C$ contains two distinct points $x,y \in C$. Since $X$ is a $T_1$ topological space, the complement $X \backslash \{ x \} \subset X$ is an open neighbourhood of $y$. There are a clopen subset $U \in \m{CO}(X)$ such that $y \in U \subset X \backslash \{ x \}$ because $\m{CO}(X)$ forms a basis of $X$. Now one obtains a decomposition $C = (C \cap U) \sqcup (C \cap (X \backslash U))$ of $C$ by clopen subsets, and each component is non-empty because it contains either $x$ or $y$. It contradicts the fact that $C$ is connected, and hence $C$ consists of a single point. Thus $X$ is totally disconnected.

On the other hand, suppose $X$ is a totally disconnected locally compact Hausdorff topological space. Take a point $x \in X$ and an open neighbourhood $U \subset X$ of $x$. To begin with, we construct an open neighbourhood $V \subset X$ of $x$ such that $\overline{V} \subset X$ is compact and $\overline{V} \subset U$. Since $X$ is locally compact, there is an open neighbourhood $W \subset X$ of $x$ such that $\overline{W}$ is compact. The closed subset $\overline{W} \backslash U \subset X$ is a compact subset disjoint from $x$, and hence since $X$ is Hausdorff, there are open subsets $O_1,O_2 \subset X$ such that $x \in O_1$, $\overline{W} \backslash U \subset O_2$, and $O_1 \cap O_2 = \emptyset$. In particular one has
\begin{eqnarray*}
  \overline{O_1 \cap W \cap U} \subset \overline{O_1} \cap \overline{W} \subset (X \backslash O_2) \cap \overline{W} = \left( X \backslash \left( \overline{W} \backslash U \right) \right) \cap \overline{W} \subset U \cap \overline{W},
\end{eqnarray*}
and hence $V \coloneqq O_1 \cap W \cap U \subset X$ is an open neighbourhood of $x$ such that $\overline{V} \subset X$ is compact and $\overline{V} \subset U$. Now set $F \coloneqq \bigcap (\m{CO}(\overline{V}) \cap N_{\overline{V},x}) \subset \overline{V}$. The subset $F \subset \overline{V}$ is an intersection of closed subsets containing $x$, and hence is a closed subset of $\overline{V}$ containing $x$. We prove $F$ is the singleton $\{ x \}$. Take closed subsets $G,H \subset \overline{V}$ with $F = G \sqcup H$, and we may and do assume $x \in G$ without loss of generality. Since $\overline{V}$ is a compact Hausdorff topological space, there are open subsets $W_1,W_2 \subset \overline{V}$ such that $G \subset W_1$, $H \subset W_2$, and $W_1 \cap W_2 = \emptyset$. The collection $\{ W_1 \sqcup  W_2 \} \cup \{ \overline{V} \backslash O \mid O \in \m{CO}(\overline{V} \cap N_{\overline{V},x}) \}$ is an open covering of the compact topological space $\overline{V}$, and hence there is a clopen subset $O \in \m{CO}(\overline{V}) \cap N_{\overline{V},x}$ such that $W_1 \cup W_2 \cup (\overline{V} \backslash O) = \overline{V}$, i.e.\ $O \subset W_1 \cup W_2$. Note that $W_1$ and $W_2$ are disjoint open subsets of the closed subspace $\overline{V} \subset X$, and hence $\overline{W_1} \cap W_2 = \emptyset$. One has
\begin{eqnarray*}
  && \left( \overline{W}_1 \cap O \right) \cap \left( \overline{V} \backslash W_1 \right) = \overline{W}_1 \cap O \cap \left( \left( \overline{V} \backslash (W_1 \cup W_2) \right) \cup (W_2 \backslash W_1) \right) \\
  & \subset & \overline{W}_1 \cap O \cap \left( \left( \overline{V} \backslash O \right) \cup W_2 \right) = \emptyset \cup \emptyset = \emptyset,
\end{eqnarray*}
and hence $\overline{W}_1 \cap O \subset W_1$. It follows $W_1 \cap O = \overline{W_1} \cap O$. Since $O \subset \overline{V}$ is a clopen subset, so is $W_1 \cap O = \overline{W}_1 \cap O$. One obtains $x \in G \subset W_1 \cap O \in \m{CO}(\overline{V}) \cap N_{\overline{V},x}$, and it implies
\begin{eqnarray*}
  H = F \cap H \cap F \subset (W_1 \cap O) \cap W_2 \cap O = \emptyset.
\end{eqnarray*}
Therefore $F$ is connected, and thus coincides with the singleton $\{ x \}$ because $X$ is totally disconnected. It follows that the collection $\{ \overline{V} \backslash O' \mid O' \in \m{CO}(\overline{V}) \cap N_{\overline{V},x} \}$ is a clopen covering of the closed subspace $\overline{V} \backslash V$ in the compact subspace $\overline{V} \subset X$, and hence there is a clopen subset $O' \in \m{CO}(\overline{V})$ of $\overline{V}$ such that $O' \subset V$. Since $V \subset X$ is an open subset, $O'$ is open in $X$. Moreover $O'$ is a closed subspace of the compact subspace $\overline{V} \subset X$, and hence $O'$ is compact. Since $X$ is Hausdorff, the compact subspace $O' \subset X$ is closed, and thus $O'$ is a clopen neighbourhood of $x$ in $X$ contained in $V$ and hence in $U$. We conclude that the set $\m{CO}(X)$ of clopen subsets forms a basis of $X$.
\end{proof}

\begin{proposition}
\label{P-point and zero set}
Suppose $X$ is a totally disconnected locally compact Hausdorff topological space and the valuation of $k$ is not trivial. Then a point $x \in X$ is a P-point of $X$ if and only if the zero locus $f^{-1}(\{ 0 \}) \subset X$ is a neighbourhood of $x$ in $X$ for any $k$-valued bounded continuous function $f \colon X \to k$ with $f(x) = 0$.
\end{proposition}

\begin{proof}
Suppose $x$ is a P-point of $X$. Take a $k$-valued bounded continuous function $f \colon X \to k$ with $f(x) = 0$. For a non-negative integer $n \in \N$, set
\begin{eqnarray*}
  U_n \coloneqq \set{x \in X}{|f(x)| < (n + 1)^{-1}}.
\end{eqnarray*}
By the argument in the proof of Lemma \ref{zero set}, $U_n$ is a clopen neighbourhood of $x$ in $X$ for any $n \in \N$ and $\bigcap_{n \in \N} U_n = f^{-1}(\{ 0 \})$. Since $x$ is a P-point of $X$, the countable intersection $f^{-1}(\{ 0 \}) \in \m{CO}_{\delta}(X)$ is a neighbourhood of $x$ in $X$.

On the other hand, suppose the zero locus $f^{-1}(\{ 0 \}) \subset X$ is a neighbourhood of $x$ in $X$ for any $k$-valued bounded continuous function $f \colon X \to k$ with $f(x) = 0$. Take a countable family $\U \subset N_{X,x}$, and we prove that $\bigcap \U \in N_{X,x}$. For a neighbourhood $U \in \U$ of $x$, by Lemma \ref{zero dimensional}, there is a clopen neighbourhood $U' \in \m{CO}(X) \cap N_{X,x}$ of $x$ contained in $U$. Therefore there is a countable refinement $\U' \subset \m{CO}(X) \cap N_{X,x}$ of $\U$ consisting of clopen neighbourhoods of $x$. One has $x \in \bigcap \U' \subset \bigcap \U$ and $\bigcap \U' \in \m{CO}_{\delta}(X)$. By Lemma \ref{zero set}, there is a $k$-valued bounded continuous function $f \colon X \to k$ such that $\bigcap \U' = f^{-1}(\{ 0 \})$, and hence by the assumption, $\bigcap \U' \subset X$ is a neighbourhood of $x$. Thus $\bigcap \U \in N_{X,x}$.
\end{proof}

We have reviewed the notion of a P-point, and studied the relation with a P-point of a zero dimensional topological space and the $k$-valued bounded continuous functions. Such a relation is easily translated by the language of a ring. Here we introduce specific two ideals of the $k$-algebra of $k$-valued bounded continuous functions.

\begin{definition}
Denote by $\m{C}(X,k)$ the commutative unital $k$-algebra of $k$-valued continuous functions on $X$ and by $\m{C}_{\ms{bd}}(X,k) \subset \m{C}(X,k)$ the $k$-subalgebra of $k$-valued bounded continuous functions on $X$.
\end{definition}

\begin{definition}
For a point $x \in X$, set
\begin{eqnarray*}
  m_{k,x} & \coloneqq & \set{f \in \m{C}_{\ms{bd}}(X,k)}{|f(x)| = 0} \\
  I_{k,x} & \coloneqq & \set{f \in \m{C}_{\ms{bd}}(X,k)}{|f(x)| = 0, f^{-1}(\{ 0 \}) \in N_{X,x}}.
\end{eqnarray*}
They are ideals of $\m{C}_{\ms{bd}}(X,k)$ and it is obvious that $I_{k,x} \subset m_{k,x}$. The evaluation map
\begin{eqnarray*}
  \m{C}_{\ms{bd}}(X,k) & \to & k \\
  f & \mapsto & f(x)
\end{eqnarray*}
induces the isomorphism $\m{C}_{\ms{bd}}(X,k)/m_{k,x} \cong_k k$ of $k$-algebras.
\end{definition}

\begin{proposition}
\label{criterion for a P-point}
Suppose $X$ is a totally disconnected locally compact Hausdorff topological space and the valuation of $k$ is not trivial. Then for a point $x \in  X$, the following are equivalent:
\begin{itemize}
\item[(i)] The point $x \in X$ is a P-point of $X$;
\item[(ii)] The equality $I_{k,x} = m_{k,x}$ holds; and
\item[(iii)] The maximal ideal $m_{k,x} \subset \m{C}_{\ms{bd}}(X,k)$ is of height $0$.
\end{itemize}
\end{proposition}

\begin{proof}
The condition (ii) is equivalent to the condition that the zero locus $f^{-1}(\{ 0 \}) \subset X$ is a neighbourhood of $x$ in $X$ for any $k$-valued bounded continuous function $f \colon X \to k$ with $f(x) = 0$, and hence is equivalent to the condition (i) by Proposition \ref{P-point and zero set}. Therefore it suffices to verify the equivalence of the conditions (ii) and (iii).

Suppose the equality $I_{k,x} = m_{k,x}$ holds. Take a prime ideal $\wp \in \m{Spec}(\m{C}_{\ms{bd}}(X,k))$ contained in $m_{k,x}$. For an element $f \in m_{k,x} = I_{k,x}$, there is a clopen neighbourhood $U \in \m{CO}(X) \cap N_{X,x}$ of $x$ contained in the zero locus $f^{-1}(\{ 0 \}) \subset X$. The $k$-valued characteristic function $1_U \colon X \to k$ on $U$ is a $k$-valued bounded continuous function on $X$, and satisfies $1_U(x) = 1 \neq 0$. Therefore $1_U \notin m_{k,x}$, and the equality $1_U f = 0 \in \wp$ guarantees $f \in \wp$. Thus $\wp = m_{k,x}$, and hence $m_{k,x}$ is of height $0$.

On the other hand, suppose $m_{k,x}$ is of height $0$. For an element $f \in m_{k,x}$, the multiplicative set $\{ f^n g \mid n \in \N, g \in \m{C}_{\ms{bd}}(X,k) \backslash m_{k,x} \} \subset \m{C}_{\ms{bd}}(X,k)$ contains $0 \in \m{C}_{\ms{bd}}(X,k)$ because $m_{k,x}$ is a minimal prime ideal. Take a non-negative integer $n \in \N$ and an element $g \in \m{C}_{\ms{bd}}(X,k) \backslash m_{k,x}$ with $f^n g = 0$. Since $g \notin m_{k,x}$, the closed subset $g^{-1}(\{ 0 \}) \subset X$ does not contain $x$, and the complement $U \coloneqq X \backslash g^{-1}(\{ 0 \})$ is an open neighbourhood of $x$. The equality $f^n g = 0$ implies that the zero locus $f^{-1}(\{ 0 \}) \subset X$ contains the open neighbourhood $U$ of $x$, and hence $f \in I_{k,x}$. Thus $m_{k,x} = I_{k,x}$.
\end{proof}

\begin{corollary}
\label{criterion for the existence of a P-point}
Suppose $X$ is a totally disconnected compact Hausdorff topological space and the valuation of $k$ is not trivial. Then the following are equivalent:
\begin{itemize}
\item[(i)] The topological space $X$ does not have a P-point;
\item[(ii)] The $k$-algebra $\m{C}(X,k)$ does not have a maximal ideal of height $0$; and
\item[(iii)] Each minimal prime ideal of the $k$-algebra $\m{C}(X,k)$ is not closed with respect to the topology given by the supremum norm.
\end{itemize}
\end{corollary}

\begin{proof}
Suppose the condition (i) holds. By the non-Archimedean generalised Stone-Weierstrass theorem, \cite{Ber} 9.2.5, the evaluation map $X \to \M_k(\m{C}(X,k))$ is surjective, where $\M_k(\m{C}(X,k))$ is the Berkovich spectrum of the $k$-algebra $\m{C}(X,k)$ endowed with the supremum norm. By the proof of \cite{Mih} 3.3, there is a canonical bijective map $\M_k(\m{C}(X,k)) \to \m{Max}(\m{C}(X,k))$ compatible with the evaluation maps, and hence any maximal ideal of $\m{C}(X,k)$ is of the form $m_{k,x}$ for some $x \in X$. Therefore by Proposition \ref{criterion for a P-point}, no maximal ideal is of height $0$: the condition (ii). Suppose the condition (ii) holds. Take a minimal prime ideal $\wp \subset \m{C}(X,k)$. Since $\m{C}(X,k)$ does not has a maximal ideal of height $0$, $\wp$ is not a maximal ideal. Since a closed prime ideal is a maximal ideal by \cite{Mih} 2.7, and hence $\wp$ is not closed: the condition (iii). Suppose the condition (iii) holds. Take a point $x \in X$. The maximal ideal $m_{k,x} \subset \m{C}(X,k)$ is closed by the completeness of $\m{C}(X,k)$ and by \cite{BGR} 1.2.4/5, and hence $m_{k,x}$ is not of height $0$ by the condition (iii). Therefore $x$ is not a P-point by Proposition \ref{criterion for a P-point}.
\end{proof}

\begin{corollary}
Suppose $X$ is a non-empty totally disconnected compact Hausdorff topological space and the valuation of $k$ is not trivial. Then the following are equivalent:
\begin{itemize}
\item[(i)] $|X| < \aleph_0$;
\item[(ii)] $|\m{Spec}(\m{C}(X,k))| < \aleph_0$;
\item[(iii)] $\dim \m{C}(X,k) = 0$; and
\item[(iv)] Any prime ideal of the $k$-algebra $\m{C}(X,k)$ is closed with respect to the topology given by the supremum norm.
\end{itemize}
\end{corollary}

\begin{proof}
Firstly, suppose the condition (i) holds. Since $X$ is a finite set, the direct product $k^X$ is a finite $k$-algebra. In particular one has $k^X$ is finite over the $k$-subalgebra $\m{C}(X,k) \subset k^X$, and hence the associated continuous map $\m{Spec}(k^X) \to \m{Spec}(\m{C}(X,k))$ is surjective. Since $\m{Spec}(k^X) = \m{Spec}(k)^{\sqcup X}$, one has $|\m{Spec}(\m{C}(X,k))| = |X| < \aleph_0$, and therefore $|\m{Spec}(\m{C}(X,k))| < \aleph_0$: the condition (ii).

Secondly, suppose the condition (ii) holds. Denote by $M_{k,X} \subset \m{Spec}(\m{C}(X,k))$ the image of the evaluation map $X \to \m{Spec}(\m{C}(X,k)) \colon x \mapsto m_{k,x}$. The embedding $\m{C}(X,k) \hookrightarrow k^X$ induces an injective map $\m{C}(X,k) \hookrightarrow k^{M_{k,X}}$ by the definition of the evaluation map, and hence $\m{C}(X,k)$ is a finite $k$-algebra because $|M_{k,x}| \leq |\m{Spec}(\m{C}(X,k))| < \aleph_0$. A finite $k$-algebra is an Artinian ring, and its Krull dimension is $0$. Therefore one has $\dim \m{C}(X,k) = 0$: the condition (iii).

Thirdly suppose the condition (iii) holds. Then any prime ideal is a maximal ideal. Therefore by the completeness of $\m{C}(X,k)$ and by \cite{BGR} 1.2.4/5, any prime ideal is closed: the condition (iv).

Finally suppose the condition (iv) holds. For a point $x \in X$, take a prime ideal $\wp \in \m{Spec}(\m{C}(X,k))$ contained in $m_{k,x}$. Since $\wp$ is closed, $\wp$ is a maximal ideal by \cite{Mih} 2.7, and hence one has $\wp = m_{k,x}$. It follows $m_{k,x}$ is of height $0$, and therefore $x$ is a P-point by Proposition \ref{criterion for a P-point}. In order to prove $|X| < \aleph_0$, assume $|X| \geq \aleph_0$. We construct a bounded continuous function $f \colon X \to k$ whose image is not a finite set. To begin with, we show that for a clopen infinite subset $U \in \m{CO}(X)$, there is a clopen subset $U' \subset \m{CO}(X)$ such that $U'$ is a non-empty subset of $U$ and the complement $U \backslash U'$ is an infinite set. Since $U$ is an infinite set, there are two distinct points $x,y \in U$, and the complement $X \backslash \{ y \} \subset X$ is an open neighbourhood of $x$ because $X$ is Hausdorff. By Lemma \ref{zero dimensional}, there is a clopen neighbourhood $V \in \m{CO}(X) \cap N_{X,x}$ of $x$ in $X$ such that $y \notin V$. Since $U$ is an infinite set, at least one of the clopen subsets $V \cap U \subset X$ and $U \backslash V \subset X$ is an infinite set. Choose one, and denote by $U' \in \m{CO}(X)$ the other one. Since $U$ is a clopen subset of $X$, $U'$ is also a clopen subset of $X$, and is not empty because it contains one of $x$ and $y$. Now for a non-negative integer $n \in \N$, we define a non-empty clopen subset $U_n \in \m{CO}(X)$ in the inductive way so that the complement $X \backslash (U_0 \sqcup \cdots \sqcup U_n)$ is an infinite set for any $n \in \N$ and the collection $\{ U_n \mid n \in \N \}$ is pairwise disjoint. When $n = 0$, there is a non-empty clopen subset $U' \in \m{CO}(X)$ such that the complement $X \backslash U'$ is an infinite set, and put $U_0 \coloneqq U' \in \m{CO}(X)$. When $n > 0$, since the complement $X \backslash (U_0 \sqcup \cdots \sqcup U_{n - 1})$ is a clopen infinite subset of $X$, there is a clopen subset $U' \in \m{CO}(X)$ such that $U'$ is a non-empty subset of $X \backslash (U_0 \sqcup \cdots \sqcup U_{n - 1})$ and the complement $(X \backslash (U_0 \sqcup \cdots \sqcup U_{n - 1})) \backslash U'$ is an infinite set. Put $U_n \coloneqq U' \in \m{CO}(X)$. One has obtained a map $U \colon \N \to \m{CO}(X) \colon n \mapsto U_n$ such that the complement $X \backslash (U_0 \sqcup \cdots \sqcup U_n)$ is an infinite set for any $n \in \N$ and the collection $\{ U_n \mid n \in \N \}$ is pairwise disjoint. Since the valuation of $k$ is not trivial, there is an element $a \in k^{\times}$ such that $0 < |a| < 1$. Define a bounded function $f \colon X \to k$ by setting
\begin{eqnarray*}
  f(x) \coloneqq
  \left\{
    \begin{array}{ll}
      a^n & (x \in U_n, {}^{\exists} n \in \N)\\
      0 & (x \notin U_n, {}^{\forall} n \in \N)
    \end{array}
  \right..
\end{eqnarray*}
Since $f$ is constant on a clopen subset $U_n \subset X$, $f$ is continuous at a point in $U_n$ for any $n \in \N$. For a point $x \in X$ with $x \notin U_n$ for any $n \in \N$ and a positive number $\epsilon > 0$, take a non-negative integer $n \in \N$ with $|a|^{n + 1} < \epsilon$. Since $U_0,\ldots,U_n \subset U$ are clopen, the subset $U \coloneqq X \backslash (U_0,\ldots,U_n)$ is an open neighbourhood of $x$. One has
\begin{eqnarray*}
  f(U) = \{ 0 \} \sqcup \set{a^m}{m \in \N, m > n} \subset \set{b \in k}{|b - f(x)| < \epsilon},
\end{eqnarray*}
and hence $f$ is continuous at $x$. Therefore $f$ is a bounded continuous function whose image $S \subset k$ is an infinite set containing $\{ a^n \mid n \in \N \}$. Take a point $b \in S$. For a point $x \in f^{-1}(\{ b \})$, since $x$ is a P-point of $X$, the pre-image $f^{-1}(\{ b \}) \subset X$, which is the zero locus  of the bounded continuous function $f - f(b)$, contains a neighbourhood of $x$ by Proposition \ref{P-point and zero set}. Therefore $f^{-1}(b) \subset X$ is open. It follows that the collection $\{ f^{-1}(b) \mid b \in S \}$ is an infinite disjoint open covering of $X$ consisting of non-empty subset, and hence it contradicts the assumption that $X$ is compact. Thus we conclude $|X| < \aleph_0$: the condition (i).
\end{proof}

Thus the purely topological notion of a P-point is interpreted in the purely ring-theoretical property of ideals of the specific ring. The ring of bounded continuous functions possesses many other informations about the topology of the underlying space. For example, the non-Archimedean analogue of Tietze's extension theorem is useful when one analyses a closed subspace of a zero-dimensional topological space.

\begin{lemma}
\label{idempotent}
Suppose $X$ is a totally disconnected locally compact Hausdorff topological space. For a compact subset $F \subset X$, the pull-back $\m{CO}(X) \to \m{CO}(F) \colon U \mapsto U \cap F$ is surjective.
\end{lemma}

\begin{proof}
Since $F$ is compact and $X$ is Hausdorff, $F$ is a closed subset of $X$. Take a clopen subset $V \in \m{CO}(F)$. Since  $V$ and $F \backslash V$ are disjoint clopen subsets of $F$, they are disjoint closed subsets in $X$. The topological space $V$ is closed in the compact topological space $F$, and hence $V$ is compact. Since $X \backslash (F \backslash V)$ is an open subset containing $V$, there is a clopen neighbourhood of $x$ contained in $X \backslash (F \backslash V)$ for each point $x \in V$ by Lemma \ref{zero dimensional}. Therefore there is a finite covering of $V$ consisting of clopen subsets of $X$ contained in $X \backslash (F \backslash V)$, and the union $U$ of the covering is a clopen subset of $X$ containing $V$ and contained in $X \backslash (F \backslash V)$. Thus one has obtained a clopen subset $U \in \m{CO}(X)$ with $U \cap F = V$.
\end{proof}

\begin{lemma}
\label{Tietze}
Suppose $X$ is a totally disconnected locally compact Hausdorff topological space. For a compact subset $F \subset X$, the restriction $\m{C}_{\ms{bd}}(X,k) \to \m{C}(F,k) \colon f \mapsto f|_F$ is surjective.
\end{lemma}

\begin{proof}
Since $F$ is compact, any $k$-valued continuous function on $F$ is bounded. Take a $k$-valued continuous function $g \colon F \to k$. For a non-negative integer $n \in \N$, we construct a $k$-valued bounded continuous functions $f_n \colon X \to k$ and $g_n \coloneqq F \to k$ in an inductive way. When $n = 0$, set $f_n \coloneqq 0 \in \m{C}_{\ms{bd}}(X,k)$ and $g_n \coloneqq g \in \m{C}(F,k)$. Suppose $n > 0$. Denote by $\U_n \subset 2^k$ the partition of $k$ into open discs of radii $n^{-1}$, which are clopen in $k$. The pre-image $g_{n - 1}^{-1}(\U_n) \coloneqq \{ g_{n - 1}^{-1}(U) \mid U \in \U_n \} \backslash \{ \emptyset \} \subset 2^F$ is a partition of $F$ into disjoint non-empty clopen subsets of $F$. Since $F$ is compact, the partition $g_{n - 1}^{-1}(\U_n)$ of $F$ by disjoint non-empty clopen subsets is a finite covering. By Lemma \ref{idempotent}, there is a family $\V_n \subset \m{CO}(X)$ of clopen subsets of $X$ such that the pull-back $\m{CO}(X) \to \m{CO}(F)$ induces a bijective map $\V_n \to g_{n - 1}^{-1}(\U_n)$. Since $g_{n - 1}^{-1}(\U_n)$ is a finite set consisting of non-empty sets, so is $\V_n$. Take a representative $a_{n,V} \in V$ for each $V \in \V_n$. The finite sum $f_n \coloneqq \sum_{V \in \V_n} g_{n - 1}(a_{n,V}) 1_V$ is a $k$-valued locally constant continuous function on $X$ with finite image, and set $g_n \coloneqq g_{n - 1} - f_n|_F \in \m{C}(F,k)$. By the construction of $f_n$, one has $\| g_n \| \leq n^{-1}$ and $\| f_n \| \leq \| g_{n - 1} \|$, where $\| \cdot \| \colon \m{C}(F,k) \to [0,\infty)$ and $\| \cdot \| \colon \m{C}_{\ms{bd}}(X,k) \to [0,\infty)$ are the supremum norms. It follows that the sequence $(f_n)_{n \in \N} \in \m{C}_{\ms{bd}}(X,k)^{\N}$ satisfies $\| f_{n + 1} \| \leq n^{-1}$ for any $n \in \N$, and hence the infinite sum $\sum_{n \in \N} f_n$ has the uniform convergence limit $f \in \m{C}_{\ms{bd}}(X,k)$ because $\m{C}_{\ms{bd}}(X,k)$ is complete with respect to the non-Archimedean norm $\| \cdot \|$. Moreover one has
\begin{eqnarray*}
  \left\| g - f|_F \right\| = \left\| g - \sum_{n \in \N} f_n|_F \right\| = \lim_{m \in \N} \left\| g_m - \sum_{n = m + 1} f_n|_F \right\| \leq \lim_{m \to \infty} m^{-1} = 0,
\end{eqnarray*}
and hence $f|_F = g$.
\end{proof}

In order to make use of the non-Archimedean Tietze's extension theorem for the theory of a P-point, we introduce the notion of an absolute value function $A \colon k \to k$, which is analogous to the absolute value $| \cdot | \colon \C \to \C \colon z \mapsto |z|$. An absolute value function sometimes helps one to avoid problems originating from the absence of the notion of the positivity in the $p$-adic world.

\begin{definition}
For a section $\sigma \colon |k| \to k$ of the norm $| \cdot | \colon k \to [0,\infty)$, denote by $A_{\sigma} \colon k \to k$ the composition of the section $\sigma$ and the norm $| \cdot |$. A map $A \colon k \to k$ is said to be an absolute value function if there is a section $\sigma \colon |k| \to k$ of the norm $| \cdot |$ such that $A_{\sigma} = A$.
\end{definition}

\begin{lemma}
An absolute value function $A \colon k \to k$ is continuous.
\end{lemma}

\begin{proof}
Take a point $a \in k$ and a positive number $\epsilon$. If $a \neq 0$, then one has $| A(a) - A(b)| = 0 < \epsilon$ for any $b \in k$ with $|b - a| < |a|$. Otherwise, one obtains $|A(b) - A(a)| = |A(b) - A(0)| = |A(b)| = |b| < \epsilon$ for any $b \in k$ with $|b - a| < \epsilon$. Thus $A$ is continuous.
\end{proof}

\begin{lemma}
For an element $u \in k$ with $|u| = 1$ and an absolute value function $A \colon k \to k$, the product $uA \colon k \to k \colon a \mapsto uA(a)$ is an absolute value function.
\end{lemma}

\begin{proof}
For the section $\sigma \colon |k| \to k$ of the norm $| \cdot |$ with $A_{\sigma} = A$, one has $uA = uA_{\sigma} = A_{u \sigma}$ for the section $u \sigma \colon |k| \to k \colon a \mapsto u \sigma(a)$ of the norm $| \cdot | \colon k \to [0,\infty)$.
\end{proof}

\begin{lemma}
For an absolute value function $A \colon k \to k$ and a $k$-valued bounded continuous function $f \colon X \to k$, the composition $f_A \coloneqq A \circ f \colon X \to k$ a $k$-valued bounded continuous function.
\end{lemma}

\begin{proof}
Since an absolute value function is continuous, it suffices to show the boundedness of the composition $A \circ f$. It is obvious because an absolute value function preserves the norm.
\end{proof}

\begin{lemma}
Suppose the residual characteristic $p \in \N$ of $k$ is not $2$. For an absolute value function $A \colon k \to k$ and $k$-valued bounded continuous functions $f,g \colon X \to k$, one has
\begin{eqnarray*}
  |f_A(x) + g_A(x)| = \max \{ |f(x)|,|g(x)| \}
\end{eqnarray*}
for any $x \in X$, and hence
\begin{eqnarray*}
  (f_A + g_A)^{-1}(\{ 0 \}) = f^{-1}(\{ 0 \}) \cap g^{-1}(\{ 0 \}).
\end{eqnarray*}
\end{lemma}

\begin{proof}
The second assertion is obvious from the first assertion, and therefore it suffices to verify $|f_A(x) + g_A(x)| = \max \{ |f(x)|,|g(x)| \}$ for any $x \in X$. Since $A$ preserves the norm, the equality holds if $|f(x)| \neq |g(x)|$. Suppose $|f(x)| = |g(x)|$. Take the section $\sigma \colon |k| \to k$ of the norm $| \cdot | \colon k \to [0,\infty)$ with $A_{\sigma} = A$. Then one obtains
\begin{eqnarray*}
  f_A(x) + g_A(x) = \sigma(|f(x)|) + \sigma(|g(x)|) = 2 \sigma(|f(x)|) = f_{A_{2 \sigma}}(|f(x)|) = |f(x)| = \max \{ |f(x)|,|g(x)| \}
\end{eqnarray*}
because $|2| = 1$ by the assumption that $p \neq 2$.
\end{proof}

\begin{definition}
A topological space $X$ is said to be Lindel\"of if any open covering of $X$ has a countable subcovering, to be $\sigma$-compact if $X$ admits a countable family $\mathscr{K} \subset 2^X$ of compact subsets with $X = \bigcup \mathscr{K}$, and to be countable at infinity if $X$ admits an increasing countable sequence $K_0 \subset K_1 \subset \cdots \subset K_n \subset \cdots \subset X$ of compact subsets with $\bigcup_{n \in \N} K_n = X$ and $K_n$ is contained in the interior of $K_{n + 1} \subset X$ for any $n \in \N$. Call such an increasing countable sequence a compact exhaustion of $X$.
\end{definition}

\begin{remark}
The following implications are obvious:
\begin{itemize}
\item[(i)] A $\sigma$-compact topological space is Lindel\"of;
\item[(ii)] A topological space countable at infinity is $\sigma$-compact; and
\item[(iii)] A topological space is countable at infinity if and only if it is locally compact and Lindel\"of.
\end{itemize}
\end{remark}

\begin{definition}
A $k$-valued continuous function $f \colon X \to k$ is said to be compact-supported if the pre-image of the subset $\{ a \in k \mid |a| > \epsilon \} \subset k$ by $f$ is relatively compact in $X$ for any $\epsilon > 0$. A compact-supported continuous function is bounded, and the subset $\m{C}_0(X,k) \subset \m{C}_{\ms{bd}}(X,k)$ of compact-supported continuous functions is an ideal.
\end{definition}

Note that the ``relatively'' can be removed because an open disc in a non-Archimedean field is closed. Namely, a $k$-valued continuous function is compact-supported if and only if the pre-image of an open disc is compact.

\begin{lemma}
A totally disconnected locally compact Hausdorff topological space $X$ is countable at infinity if and only if $X$ admits a compact exhaustion consisting of clopen subsets. Call such a compact exhaustion a compact clopen exhaustion
\end{lemma}

\begin{proof}
The direct assertion is trivial, and it suffices to show the inverse assertion. Suppose $X$ is countable at infinity. Take a compact exhaustion $K_0 \subset K_1 \subset \cdots \subset X$ of $X$. For any $n \in \N$, there is a clopen covering $\U_n \subset \m{CO}(X)$ of $K_n$ contained in the interior of $K_{n + 1}$, by Lemma \ref{zero dimensional}. Since $K_n$ is compact, $\U_n$ has a finite subcovering $\V_n \subset \U_n$, and the union $V_n \coloneqq \bigcup \V_n$ is a compact clopen subset of $X$ containing $K_n$ and contained in the interior of $K_{n + 1}$. Therefore the increasing sequence $V_0 \subset V_1 \subset \cdots \subset X$ is a compact clopen exhaustion.
\end{proof}

\begin{definition}
A complete valuation field $k$ is said to be a local field if $k$ is a complete discrete valuation field and its residue field is a finite field.
\end{definition}

\begin{lemma}
\label{compact-supported}
Suppose $X$ is a totally disconnected locally compact Hausdorff topological space and $k$ is a local field. A topological space $X$ is countable at infinity if and only if there is a compact-supported continuous function $f \colon X \to k$ with no zero.
\end{lemma}

\begin{proof}
Suppose $X$ is countable at infinity. Take a compact clopen exhaustion $K_0 \subset K_1 \subset \cdots \subset X$. Since $k$ is a discrete valuation field, there is an element $a \in k$ such that $0 < |a| < 1$. Then the infinite sum $\sum_{n \in \N} a^n 1_{K_n}$ has the uniform convergence limit $f \in \m{C}_{\ms{bd}}(X,k)$, and it is obvious that $|f(x)| = |a|^n$ for any $x \in X$ and the smallest integer $n \in \N$ with $x \in K_n$. It follows $f \in \m{C}_0(X,k)$ and $f$ has no zero.

On the other hand, suppose there is a compact-supported continuous function $f \colon X \to k$ with no zero. Then for each non-negative integer $n \in \N$, the subset $K_n \coloneqq \{ x \in X \mid |f(x)| > (n + 1)^{-1} \} \subset X$ is a clopen subset by the argument in the proof of Lemma \ref{zero set}. Moreover, the clopen subset $K_n \subset X$ is compact by the definition of a compact-supported continuous function, and the increasing countable sequence $K_0 \subset K_1 \subset \cdots \subset X$ is a compact clopen exhaustion.
\end{proof}

\begin{proposition}
\label{restriction of zero sets}
Suppose $X$ is a totally disconnected compact Hausdorff topological space. For a closed subset $F \subset X$ whose complement $X \backslash F \subset X$ is countable at infinity, the pull-back $2^X \to 2^F \colon U \mapsto U \cap F$ induces a surjective map $\m{CO}_{\delta}(X) \twoheadrightarrow \m{CO}_{\delta}(F)$.
\end{proposition}

\begin{proof}
It is obvious that the pull-back induces a map $\m{CO}_{\delta}(X) \twoheadrightarrow \m{CO}_{\delta}(F)$, and it suffices to show the surjectivity of it. More strongly, we verify that $\m{CO}_{\delta}(F) \subset \m{CO}_{\delta}(X)$. Take a $\m{CO}_{\delta}$-subset $V \in \m{CO}_{\delta}(F)$ of $F$. By Lemma \ref{zero set}, there is a bounded continuous function $g \colon F \to \Q_3$ such that $g^{-1}(\{ 0 \}) = V$, and by Lemma \ref{Tietze}, $g$ admits a bounded continuous extension $\tilde{g} \colon X \to \Q_3$. Since $X \backslash F$ is countable at infinity, there is a compact-supported continuous function $f \colon X \backslash F \to \Q_3$ with no zero by Lemma \ref{compact-supported}. Denote by $\tilde{f} \colon X \to \Q_3$ the extension of $f$ by $0$ outside $X \backslash F$. Since $X \backslash F \subset X$ is a open subset, the extension $\tilde{f}$ is continuous at a point of $X \backslash F$. For a point $x \in F$ and a positive number $\epsilon > 0$, denote by $U \subset X$ the pre-image of $\{ a \in k \mid |a| < \epsilon \}$ by $\tilde{f}$. Since $\tilde{f}(F) \subset \{ 0 \}$, the complement $X \backslash U$ is the pre-image of $\{ a \in k \mid |a| \geq \epsilon \}$ by $f$, and is compact because $f$ is compact-supported. Therefore $X \backslash U$ is a compact subset of the Hausdorff topological space $X$, and is closed. It follows $U \subset X$ is a open subset, and $\tilde{f}$ is continuous at $x \in F$. Thus $\tilde{f} \colon X \to k$ is continuous. By definition, one has $\tilde{f}^{-1}(\{ 0 \}) = F$. For a section $\sigma \colon |k| \to k$ of the norm $| \cdot |$, one obtains
\begin{eqnarray*}
  V = g^{-1}(\{ 0 \}) = F \cap \tilde{g}^{-1}(\{ 0 \}) = \tilde{f}^{-1}(\{ 0 \}) \cap \tilde{g}^{-1}(\{ 0 \}) = (\tilde{f}_{A_{\sigma}} + \tilde{g}^{-1}_{A_{\sigma}})^{-1}(\{ 0 \}),
\end{eqnarray*}
and thus $V \in \m{CO}_{\delta}(X)$ by Lemma \ref{zero set}.
\end{proof}

\begin{corollary}
Suppose $X$ is a totally disconnected compact Hausdorff topological space. For a closed subset $F \subset X$ whose complement is countable at infinity, a point $x \in F$ is a P-point of $F$ if and only if $x$ is a P-point of $X$.
\end{corollary}

\begin{proof}
It is straightforward from Lemma \ref{zero set}, Proposition \ref{P-point and zero set}, and Proposition \ref{restriction of zero sets}.
\end{proof}

\section{Totally Disconnected Boundary}
\label{Totally Disconnected Boundary}

In this section, we observe the boundary $\partial_{\ms{TD}} X$ of a topological space $X$ in the universal totally disconnected Hausdorff compactification $\iota_{\ms{TD},X} \colon X \to \m{TD}(X)$, which is universal in the category of totally disconnected compact Hausdorff topological spaces equipped with continuous maps from $X$. The existence of the universal totally disconnected Hausdorff compactification is well-known, and we studied the four constructions $\m{UF}(X)$, $\M_k(\m{C}_{\ms{bd}}(X,k))$, $\m{Max}(\m{C}_{\ms{bd}}(X,k))$, and $\m{SC}_k(X)$ of the universal totally disconnected Hausdorff compactification $\m{TD}(X)$ in our previous paper \cite{Mih}. As the real analysis is useful in the study of the Stone-$\check{\m{C}}$ech compactification, the $p$-adic analysis works well in that of the universal totally disconnected Hausdorff compactification.

\begin{definition}
The universal totally disconnected Hausdorff compactification of $X$ is a totally disconnected compact Hausdorff topological space $\m{TD}(X)$ equipped with a continuous map $\iota_{\ms{TD},X} \colon X \to \m{TD}(X)$ such that for any totally disconnected compact Hausdorff topological space $Y$ and a continuous map $\phi \colon X \to Y$, there uniquely exists a continuous map $\m{TD}(\phi) \colon \m{TD}(X) \to Y$ such that $\phi = \m{TD}(\phi) \circ \iota_{\ms{TD},X}$. The universal totally disconnected Hausdorff compactification $\m{TD}(X)$ of $X$ uniquely exists up to unique homeomorphism over $X$.
\end{definition}

Choosing one of the constructions of $(\m{TD}(X),\iota_{\ms{TD},X})$, we fix the functor $\m{TD}$ in this paper without considering homeomorphism classes in the category of topological spaces. 

\begin{definition}
Denote by $\partial_{\ms{TD}} X$ the subset $\m{TD}(X) \backslash \iota_{\ms{TD},X}(X)$, and call it the totally disconnected boundary of $X$.
\end{definition}

\begin{proposition}
\label{open immersion}
Suppose $X$ is a totally disconnected Hausdorff topological space countable at infinity. Then the structure continuous map $\iota_{\ms{TD},X} \colon X \to \m{TD}(X)$ is a homeomorphism onto the open dense subset, and the boundary $\partial_{\ms{TD}} X$ is a $\m{CO}_{\delta}$-subset of $\m{TD}(X)$.
\end{proposition}

\begin{proof}
We verified that $\iota_{\ms{TD},X}$ is a homeomorphism onto the dense image in \cite{Mih} 6.2, and hence it suffices to show that the image $\iota_{\ms{TD},X}(X) \subset \m{TD}(X)$ is open for the first assertion. Take a compact clopen exhaustion $K_0 \subset K_1 \subset \cdots \subset X$. Consider the $\Q_2$-valued function
\begin{eqnarray*}
  f \colon X & \to & \Q_2 \\
  x & \mapsto & 2^{\min \set{n \in \N}{x \in K_n}}.
\end{eqnarray*}
Since $f$ is a bounded locally constant function, it is a bounded continuous function. Moreover it is a compact-supported continuous function with no zero because $K_n$ is compact for any $n \in \N$, and the image $f(X) \subset \Q_2$ is contained in the subring $\Z_2 \subset \Q_2$ of integral elements. Since $\Z_2$ is a totally disconnected compact Hausdorff topological space, there uniquely exists a continuous extension $\m{TD}(f) \colon \m{TD}(X) \to \Z_2$ of $f$ by the universality of $\m{TD}(X)$. Take a point $x \in \partial_{\ms{TD}} X$. For a neighbourhood $U \in N_{\ms{TD}(X),x}$ of $x$ in $\m{TD}(X)$ and a non-negative integer $n \in \N$, consider the subset $U \backslash \iota_{\ms{TD},X}(K_n) \subset \m{TD}(X)$. Since $K_n$ is compact and $\m{TD}(X)$ is Hausdorff, the image $\iota_{\ms{TD},X}(K_n)$ is closed in $\m{TD}(X)$ and hence $U \backslash \iota_{\ms{TD},X}(K_n) \in N_{\ms{TD}(X),x}$. Since the image $\iota_{\ms{TD},X}(X) \subset \m{TD}(X)$ is dense, the image of $U \backslash \iota_{\ms{TD}}(K_n)$ by the continuous function $\m{TD}(f)$ is contained in the closure of the subset
\begin{eqnarray*}
  && \m{TD}(f) \left( \iota_{\ms{TD},X}(X) \cap (U \backslash \iota_{\ms{TD},X}(K_n)) \right) = \m{TD}(f)(\iota_{\ms{TD},X}(X \backslash K_n)) = f(X \backslash K_n) \\
  & \subset & \set{2^m}{m \in \N, m > n} \subset 2^{n + 1} \Z_2.
\end{eqnarray*}
Therefore $\m{TD}(f)(x) \in \bigcap_{n \in \N} 2^{n + 1} \Z_2 = \{ 0 \}$, and hence $\partial_{\ms{TD}} X = \m{TD}(f)^{-1}(\{ 0 \}) \in \m{CO}_{\delta}(\m{TD}(X))$. It follows $\iota_{\ms{TD},X}(X) = \m{TD}(X) \backslash \partial_{\ms{TD}} X = \m{TD}(f)^{-1}(k^{\times})$ and the image $\iota_{\ms{TD},X}(X) \subset \m{TD}(X)$ is open.
\end{proof}

\begin{corollary}
Suppose $X$ is a totally disconnected Hausdorff topological space countable at infinity. The totally disconnected boundary $\partial_{\ms{TD}} X$ is a totally disconnected compact Hausdorff topological space. If $k$ is a finite field endowed with the trivial norm or a local field, the restrictions $\m{C}(\m{TD}(X),k) \to \m{C}_{\ms{bd}}(X,k)$ and $\m{C}(\m{TD}(X),k) \to \m{C}(\partial_{\ms{TD}} X,k)$ induce a surjective $k$-algebra homomorphism $\m{C}_{\ms{bd}}(X,k) \twoheadrightarrow \m{C}(\partial_{\ms{TD}} X,k)$.
\end{corollary}

\begin{proof}
The first assertion is obvious by Proposition \ref{open immersion}. The restriction $\m{C}(\m{TD}(X),k) \to \m{C}_{\ms{bd}}(X,k)$ is an isomorphism of $k$-algebras by the universality of $\m{TD}(X)$, and the restriction $\m{C}(\m{TD}(X),k) \to \m{C}(\partial_{\ms{TD}} X,k)$ is surjective by Lemma \ref{Tietze}. 
\end{proof}

\begin{corollary}
\label{continuous functions on a separable space}
Suppose $X$ is a separable totally disconnected Hausdorff topological space countable at infinity and $k$ is a finite field endowed with the trivial norm or a local field. Then the cardinality of the $k$-algebras $\m{C}_{\ms{bd}}(X,k)$ and $\m{C}(\partial_{\ms{TD}} X,k)$ is at most $2^{\aleph_0}$.
\end{corollary}

\begin{proof}
Since there is a surjective map $\m{C}_{\ms{bd}}(X,k) \twoheadrightarrow \m{C}(\partial_{\ms{TD}} X,k)$, it suffices to show that the cardinality of $\m{C}_{\ms{bd}}(X,k)$ of at most $2^{\aleph_0}$. Take a countable dense subset $Y \subset X$. The density of $Y \subset X$ guarantees the restriction
\begin{eqnarray*}
  \m{C}_{\ms{bd}}(X,k) & \to & k^Y \\
  f & \mapsto & (f(y))_{y \in Y}
\end{eqnarray*}
is injective. Since $k$ is a finite field or a local field, one has $|k| \leq 2^{\aleph_0}$. It follows
\begin{eqnarray*}
  |\m{C}_{\ms{bd}}(X,k)| \leq \left| k^Y \right| \leq \left( 2^{\aleph_0} \right)^{\aleph_0} = 2^{\aleph_0}.
\end{eqnarray*}
\end{proof}

We finish the observation of the totally disconnected boundary by establishing the intersection property for $\m{CO}_{\delta}$-subsets. The universal totally disconnected Hausdorff compactification is a huge compactification as there is the canonical homeomorphism $\beta \N \cong \m{TD}(\N)$. Therefore the totally disconnected boundary contains monstrously great amount of points, and some points satisfy the good intersection property. This is obviously an analogue of the corresponding result for the Stone-$\check{\m{C}}$ech compactification. See \cite{Dal} 4.2.21.

\begin{proposition}
\label{zero set in the boundary}
Suppose $X$ is a totally disconnected Hausdorff topological space countable at infinity. Then a non-empty $\m{CO}_{\delta}$-subset of $\partial_{\ms{TD}} X$ contains an interior point.
\end{proposition}

\begin{proof}
Let $F \subset \partial_{\ms{TD}} X$ be a non-empty $\m{CO}_{\delta}$-subset. By Proposition \ref{restriction of zero sets}, $F$ is a $\m{CO}_{\delta}$-subset of $\m{TD}(X)$. Therefore by Lemma \ref{zero set}, there is a continuous function $f \colon \m{TD}(X) \to \Q_5$ such that $f^{-1}(\{ 0 \}) = F$. Since $\iota_{\ms{TD},X}(X) \subset \m{TD}(X)$ is a dense subset contained in the complement $\m{TD}(X) \backslash F$, the subset $F \subset \m{TD}(X)$ contains no interior point. In particular $f^{-1}(5^m \Z_5) \backslash F \neq \emptyset$ for any $m \in \N$, and hence the subset $S \coloneqq \{ m \in \N \mid f^{-1}(5^m \Z_5 \backslash 5^{m + 1} \Z_5) \neq \emptyset \} \subset \N$ is a countable infinite set. Take a compact clopen exhaustion $K_0 \subset K_1 \subset \cdots \subset X$. For a non-negative integer $n \in \N$, denote by $m_n \in \N$ the $(n + 1)$-th smallest element in the countable infinite subset $S$, and by $l_n \in \N$ the smallest element in the non-empty subset $\{ l \in \N \mid f^{-1}(5^{m_n} \Z_5 \backslash 5^{m_n + 1} \Z_5) \cap K_l \neq \emptyset \}$. Set $U_n \coloneqq K_{l_n} \cap (\iota_{\ms{TD},X} \circ f)^{-1}(5^{m_n} \Z_5 \backslash 5^{m_n + 1} \Z_5) \subset X$, where we formally set $K_{-1} \coloneqq \emptyset$. Since $\iota_{\mt{TD},X} \circ f$ is continuous, $U_n$ is a clopen subset of $X$, and therefore the union $U \coloneqq \bigcup_{n \in \N} U_n \subset X$ is open. Moreover one has
\begin{eqnarray*}
  && X \backslash U = \left( \bigsqcup_{m \in S} (\iota_{\ms{TD},X} \circ f)^{-1}(5^m \Z_5 \backslash 5^{m + 1} \Z_5) \right) \backslash \left( \bigsqcup_{n \in \N} (K_{l_n} \cap (\iota_{\ms{TD},X} \circ f)^{-1}(5^{m_n} \Z_5 \backslash 5^{m_n + 1} \Z_5)) \right) \\
  & = & \bigsqcup_{n \in \N} \left( X \backslash \left( (\iota_{\ms{TD},X} \circ f)^{-1}(5^{m_n} \Z_5 \backslash 5^{m_n + 1} \Z_5) \cup K_{l_n} \right) \right),
\end{eqnarray*}
and hence $X \backslash U$ is open. Therefore $U$ is a clopen subset of $X$. Since the $k$-valued characteristic function $1_U \colon X \to k$ on $U$ is a continuous map whose image is contained in the totally disconnected compact Hausdorff topological space $\Z_5$, there uniquely exists a continuous extension $g \colon \m{TD}(X) \to k$ by the universal property of $\m{TD}(X)$. Since $\iota_{\ms{TD},X}(X) \subset \m{TD}(X)$ is dense and the image of $1_U$ is contained in the closed subset $\{ 0,1 \} \subset \Z_5$, the image of the extension $g$ is also contained in $\{ 0,1 \} \subset \Z_5$. Therefore $g$ is a $k$-valued characteristic function on a clopen subset $V \subset \m{TD}(X)$ with $V \cap \iota_{\ms{TD},X}(X) = \iota_{\ms{TD},X}(U)$. Since $U$ is a disjoint union of countably infinitely many non-empty clopen subsets of $X$, it is not compact, and hence the image $\iota_{\ms{TD},X}(U)$ is not closed in the compact topological space $\m{TD}(X)$. Thus the support $V \subset \m{TD}(X)$ intersects with the boundary $\partial_{\ms{TD}} X$. Moreover we prove that $V \cap \partial_{\ms{TD}} X \subset F$. Take a point $x \in V \cap \partial_{\ms{TD}} X$. For a non-negative integer $n \in \N$, since $K_{l_n}$ is compact and $(\iota_{\ms{TD},X} \circ f)^{-1}(5^m \Z_5 \backslash 5^{m + 1} \Z_5) \subset X$ is clopen, the image $\iota_{\ms{TD},X}(U_n) \subset \m{TD}(X)$ is compact, and is a closed subset of the Hausdorff topological space $\m{TD}(X)$. Therefore the subset $V \backslash \bigsqcup_{m = 0}^{n} U_n \subset \m{TD}(X)$ is an open neighbourhood of $x$. Since the image $\iota_{\ms{TD},X}(X) \subset \m{TD}(X)$ is dense, the subset $\iota_{\ms{TD},X}(X) \cap (V \backslash \bigsqcup_{m = 0}^{n} U_n) = \bigsqcup_{m = n + 1}^{\infty} U_n$ is dense in $V \backslash \bigcup_{m = 0}^{n} U_n$. Therefore the image of  $V \backslash \bigcup_{m = 0}^{n} U_n$ by $f$ is contained in the closure of $f(\bigsqcup_{m = n + 1}^{\infty} U_n) \subset 5^{m_n} \Z_5$ in $\Z_5$, and hence $x \in f^{-1}(\{ 0 \}) = F$. It follows $V \cap \partial_{\ms{TD}} X \subset F$, and the non-empty $\m{CO}_{\delta}$-subset $F \subset \partial_{\ms{TD}} X$ contains an interior point.
\end{proof}

\section{Independence}
\label{Independence}

It is known that the existence of a P-point on the maximal boundary $\partial_{\beta} X \coloneqq \beta X \backslash \iota_{\beta,X}(X)$ in the Stone-$\check{\m{C}}$ech compactification $\iota_{\beta} \colon X \to \beta X$ is provable under the continuum hypothesis for a topological space $X$ in a certain class. The proof relies on the real analysis, and hence the same method is invalid for another compactification. In this section, we prove the counterpart for the universal totally disconnected Hausdorff compactification $\iota_{\ms{TD},X} \colon X \to \m{TD}(X)$ by the use of the $p$-adic analysis. It yields the independence of the axiom of $\m{ZFC}$ and the hypothesis that rings in a certain class have a maximal ideal of height $0$ through the theory of a P-point. Before studying the existence of a P-point in the totally disconnected boundary, we establish a relation with the intersection property for clopen subsets and that for dense open subsets. This is obviously analogous to Baire category theorem for a locally compact Hausdorff topological space and \cite{Dal} 4.2.20.

\begin{definition}
Denote by $D(X) \subset 2^X$ the subset of dense open subsets in $X$.
\end{definition}

\begin{definition}
For an ordinal $\alpha$, denote by $\alpha + 1$ the successor $\alpha \cup \{ \alpha \}$ of $\alpha$ as an ordinal.
\end{definition}

\begin{definition}
For a family $\mathscr{F} \subset 2^X$ and a cardinality $\alpha$, set
\begin{eqnarray*}
  \mathscr{F}_{\alpha} \coloneqq \set{\bigcap_{S \in \Sigma} S}{\Sigma \subset \mathscr{F}, |\Sigma| \leq \alpha} = \set{\bigcap_{\beta \in \alpha}V_{\beta}}{V = (V_{\beta})_{\beta \in \alpha} \in \mathscr{F}^{\alpha}}.
\end{eqnarray*}
\end{definition}

\begin{lemma}
\label{category}
Suppose $X$ is a non-empty totally disconnected compact Hausdorff topological space. For a cardinality $\alpha$, if any non-empty subset of $X$ in the family $\m{CO}(X)_{\beta} \subset 2^X$ has an interior point for any cardinality $\beta < \alpha$, then the family $D(X)_{\alpha} \subset 2^X$ does not contain $\emptyset$.
\end{lemma}

\begin{proof}
If $\alpha < \aleph_0$, the assertion is trivial because $X \neq \emptyset$, and hence it suffices to consider the case $\alpha \geq \aleph_0$. Take an element $U \in D(X)_{\alpha}$ and a presentation $U = \bigcap_{\beta \in \alpha} V_{\beta}$ by a map $V \colon \alpha \to D(X) \colon \beta \mapsto V_{\beta}$. For an ordinal $\beta \in \alpha + 1$, denote by $R_{\beta} \subset \m{CO}(X)^{\beta}$ the set of order-preserving maps from $\beta$ to $\m{CO}(X)$ with respect to the inclusion of clopen subsets, and set
\begin{eqnarray*}
  R \coloneqq \bigcup_{\beta \in \alpha + 1} \set{V' \in R_{\beta}}{\emptyset \neq V'_{\gamma} \subset V_{\gamma}, {}^{\forall} \gamma \in \beta}.
\end{eqnarray*}
For elements $V',V'' \in R$, we write $V' \leq V''$ if the domain of $V'$ is contained in that of $V''$ and $V'$ is the restriction of $V''$. Then the binary relation $\leq$ is an order on $R$. Since $|\emptyset| < \alpha$ and $X \in D(X) \neq \emptyset$, one has $\emptyset \in \{ \emptyset \} = D(X)^{\emptyset} = R_{\emptyset} \subset R$. Take a totally ordered subset $L \subset R$. Then the union $\beta_L$ of the domains of maps in $L$ is an ordinal contained in $\alpha + 1$, and maps in $L$ have the unique order-preserving extension $V'(L) \colon \beta_L \to \m{CO}(X)$ with $V'(L)_{\gamma} \subset V_{\gamma}$ for any $\gamma \in \beta_L$ by the definition of $\leq$. Therefore $L$ has an upper bound. It follows that $R$ is a non-empty inductively ordered set. Take a maximal element $V' \in R$, and let $\beta \in \alpha + 1$ be the domain of $V'$. Since $X$ is compact and $\{ V'_{\gamma} \mid \gamma \in \beta \} \subset \m{CO}(X)$ is a decreasing family of non-empty closed subset, the intersection $\bigcap_{\gamma \in \beta} V'_{\gamma} \subset X$ is non-empty. In order to prove $\beta = \alpha$, assume $\beta \in \alpha$. For an ordinal $\gamma \in \beta$, set $V''_{\gamma} \coloneqq V'_{\gamma} \in \m{CO}(X)$. First suppose $\beta$ is a successor ordinal, let $\epsilon \in \beta$ be the predecessor ordinal of $\beta$. Since $V_{\beta} \in D(X)$ is a dense open subset of $X$, the intersection $V_{\beta} \cap V''_{\epsilon}$ contains an interior point $x \in X$, and hence there is a clopen neighbourhood $V''_{\beta} \in \m{CO}(X) \cap N_{X,x}$ of $x$ contained in $V_{\beta} \cap V''_{\epsilon}$ by Lemma \ref{zero dimensional}. One has obtained a map
\begin{eqnarray*}
  V'' \colon \beta + 1 & \to & \m{CO}(X) \\
  \gamma & \mapsto & V''_{\gamma}
\end{eqnarray*}
extending $V'$. It satisfies $V''_{\gamma} \subset V'_{\gamma}$ for any $\gamma \in \beta + 1$ and is order-preserving because $\epsilon$ is the greatest ordinal in $\beta$. Therefore one has $V'' \in R$, $V' \leq V''$, and $V'' \neq V'$. It contradicts the maximality of $V'$. Next suppose $\beta$ is a limit ordinal. Since $\alpha$ is a cardinality containing $\beta$, one has $\aleph_{\beta} \coloneqq |\beta| < \alpha$. The non-empty subset $\bigcap_{\gamma \in \beta} V''_{\gamma} = \bigcap_{\gamma \in \beta} V'_{\gamma} \subset X$ is contained in $\m{CO}(X)_{\aleph_{\beta}}$, and it contains an interior point $x \in X$ by the assumption in the assertion. Since $V'_{\beta} \in D(X)$ is a dense open subset of $X$, the intersection $V'_{\beta} \cap \bigcap_{\gamma \in \beta} V''_{\gamma}$ contains a neighbourhood of $x$, and it follows that there is a clopen neighbourhood $V''_{\beta} \in \m{CO}(X) \cap N_{X,x}$ of $x$ contained in $V'_{\beta} \cap \bigcap_{\gamma \in \beta} V''_{\gamma}$ by Lemma \ref{zero dimensional}. One has obtained a map
\begin{eqnarray*}
  V'' \colon \beta + 1 & \to & \m{CO}(X) \\
  \gamma & \mapsto & V''_{\gamma}
\end{eqnarray*}
extending $V'$. It satisfies $V''_{\gamma} \subset V'_{\gamma}$ for any $\gamma \in \beta + 1$ and is order-preserving by construction. Therefore one has $V'' \in R$, $V' \leq V''$, and $V'' \neq V'$. It contradicts the maximality of $V'$. Thus the ordinal $\beta$ is not contained in $\alpha$. Since $\beta \in \alpha + 1 = \alpha \cup \{ \alpha \}$, one obtains $\beta = \alpha$. It follows
\begin{eqnarray*}
  \emptyset \neq \bigcap_{\gamma \in \alpha} V'_{\gamma} \subset \bigcap_{\gamma \in \alpha} V_{\gamma} = U,
\end{eqnarray*}
and we conclude that $U$ is non-empty.
\end{proof}

\begin{proposition}
\label{category 2}
Suppose $X$ is a totally disconnected non-compact Hausdorff topological space countable at infinity. Then the family $D(\partial_{\ms{TD}} X)_{\aleph_1} \subset 2^{\partial_{\mt{TD}} X}$ does not contain $\emptyset$.
\end{proposition}

\begin{proof}
It is straightforward by Lemma \ref{category} and Proposition \ref{zero set in the boundary}. The boundary $\partial_{\ms{TD}} X$ is not empty because $X$ is non-compact and $\iota_{\ms{TD},X} \colon X \to \m{TD}(X)$ is a homeomorphism onto the image.
\end{proof}

Now we verify the existence of a P-point in the totally disconnected boundary of a non-empty separable totally disconnected non-compact Hausdorff topological space countable at infinity such as $\N$ under the continuum hypothesis. In particular it contains the existence of a P-point in $\beta \N \backslash \N$ under the continuum hypothesis. This result for the totally disconnected boundary is quite analogous to that for the maximal boundary, and the proof here is similar with that for the maximal boundary. See \cite{Dal} 4.2.23.

\begin{theorem}
Suppose $X$ is a separable totally disconnected non-compact Hausdorff topological space countable at infinity. Then the totally disconnected boundary $\partial_{\ms{TD}} X$ has a P-point under the continuum hypothesis.
\end{theorem}

\begin{proof}
By the continuum hypothesis, the cardinality of $\m{C}(\partial_{\ms{TD}} X,\Q_7)$ is at most $\aleph_1$ by Corollary \ref{continuous functions on a separable space}. For a $\Q_7$-valued continuous function $f \colon \m{C}(\partial_{\ms{TD}} X,\Q_7)$, denote by $U_f \subset 
\partial_{\ms{TD}} X$ the union of the complement $X \backslash f^{-1}(\{ 0 \})$ and the interior of $f^{-1}(\{ 0 \}) \subset \partial_{\ms{TD}} X$. In particular the complement $\partial_{\ms{TD}} X \backslash U_f \subset \partial_{\ms{TD}} X$ has no interior point because it is the boundary of the closed subset $f^{-1}(\{ 0 \}) \subset \partial_{\ms{TD}} X$, and hence $U_f$ is a dense open subset of $\partial_{\ms{TD}} X$. Therefore the intersection $F \coloneqq \bigcap_{f \in \ms{C}(\partial_{\mt{TD}} X,\Q_7)} U_f \subset \partial_{\ms{TD}} X$ is non-empty by Proposition \ref{category 2}, and take a point $x \in F$. Now for a $\Q_7$-valued continuous function $f \colon \partial_{\ms{TD}} X \to \Q_7$ with $f(x) = 0$, $x$ is contained in the interior of the zero locus $f^{-1}(\{ 0 \}) \subset \partial_{\ms{TD}} X$ by the definition of $F$, and hence $m_{\Q_7,x} = I_{\Q_7,x}$. Therefore we conclude that $x \in \partial_{\ms{TD}} X$ is a P-point by Proposition \ref{criterion for a P-point}.
\end{proof}

Since the natural numbers $\N$ is a separable totally disconnected non-compact Hausdorff topological space countable at infinity, and since the universal totally disconnected Hausdorff compactification of a discrete topological space coincides with the Stone-$\check{\m{C}}$ech compactification of it, the result contains the well-known fact for the existence of a P-point in the maximal boundary $\beta \N \backslash \N$.

\begin{corollary}
The maximal boundary $\partial_{\beta} \N \coloneqq \beta \N \backslash \N$ has a P-point under the continuum hypothesis.
\end{corollary}

Using the interpretation of the purely topological notion of a P-point to the purely algebraic notion of ideals by Corollary \ref{criterion for the existence of a P-point}, one obtains the existence of the maximal ideal of height $0$.

\begin{corollary}
\label{existence of a maximal minimal ideal}
Suppose $X$ is a separable totally disconnected non-compact Hausdorff topological space countable at infinity and the valuation of $k$ is not trivial. Then the $k$-algebra $\m{C}(\partial_{\ms{TD}} X,k)$ has a maximal ideal of height $0$ under the continuum hypothesis.
\end{corollary}

We have verified that the existence of a maximal ideal of height $0$ for a specific ring under the continuum hypothesis. It yields the existence of such an ideal is unprovable under the axiom of $\m{ZFC}$, provided the axiom of $\m{ZFC}$ is consistent. Conversely using Shelah's theory on a proper forcing, \cite{She} VI, it is easily seen that the existence of such an ideal is not provable either under the axiom of $\m{ZFC}$, provided the axiom of $\m{ZFC}$ is consistent.

\begin{theorem}
\label{main theorem}
If the axiom of $\m{ZFC}$ is consistent, it is independent of the axiom of $\m{ZFC}$ that the $k$-algebra $\m{C}(\partial_{\ms{TD}} X,k)$ of $k$-valued continuous functions on the totally disconnected boundary $\partial_{\ms{TD}} X$ has a maximal ideal of height $0$ for any separated totally disconnected non-compact Hausdorff topological space $X$ countable at infinity and any local field $k$.
\end{theorem}

\begin{proof}
Denote by $\m{CH}$ the continuum hypothesis, and by $\m{MMI}$ the existence of a maximal ideal of height $0$ in the $k$-algebra $\m{C}(\partial_{\ms{TD}} X,k)$ for any separated totally disconnected non-compact Hausdorff topological space $X$ countable at infinity and any local field $k$. We simply write $\m{ZFC}$ instead of the axiom of $\m{ZFC}$ for short. By Corollary \ref{existence of a maximal minimal ideal}, one has
\begin{eqnarray*}
  \m{ZFC}, \m{CH} \vdash \m{MMI}.
\end{eqnarray*}
Since $\m{CH}$ is independent of and consistent with $\m{ZFC}$ provided the consistency of $\m{ZFC}$, one has
\begin{eqnarray*}
  \m{ZFC} \vdash \hspace{-.75em}/ \ \neg \m{MMI}.
\end{eqnarray*}
On the other hand, there is a model of $\m{ZFC}$ where $\partial_{\ms{TD}} \N = \m{TD}(\N) \backslash \N \cong \beta \N \backslash \N$ has no P-point by \cite{She} VI.4.8. In such a model, there is no maximal ideal of height $0$ in the $\Q_{11}$-algebra $\m{C}(\partial_{\ms{TD}} \N,\Q_{11})$ by Corollary \ref{criterion for the existence of a P-point}. It follows
\begin{eqnarray*}
  \m{ZFC} \models \hspace{-.90em}/ \ \m{MMI}
\end{eqnarray*}
and hence
\begin{eqnarray*}
  \m{ZFC} \vdash \hspace{-.75em}/ \ \m{MMI}.
\end{eqnarray*}
Thus $\m{MMI}$ is independent of $\m{ZFC}$.
\end{proof}

\vspace{0.4in}
\addcontentsline{toc}{section}{Acknowledgements}
\noindent {\Large \bf Acknowledgements}
\vspace{0.1in}

I am deeply grateful to Professor T.\ Tsuji for his gracious teaching. He helps me to study various themes in mathematics. I would like to appreciate daily discussions with my great friends. I am thankful to my family for their boundless affection. I am extremely indebted to them for all of my success.

\end{document}